\documentclass[reqno,12pt]{amsart}

\newtheorem{Thm}{Theorem}[section]
\newtheorem{Prop}[Thm]{Proposition}
\newtheorem{Lem}[Thm]{Lemma}
\newtheorem{Cor}[Thm]{Corollary}
\newtheorem{Prob}[Thm]{Problem}
\theoremstyle{definition}
\newtheorem{Rem}[Thm]{Remark}

\usepackage{enumerate}
\newcommand{\vect}[1]{\boldsymbol{#1}}

\title{Planar lattices and equilateral odd-gons}
\author{Akira Iino}
\author{Masashi Sakiyama}
\thanks{This version incorporates the corrections published in Yokohama Math.~J. {\bf 71}, 61--62 (2025). The main results of the paper remain unchanged.}


\date{\today}

\address{Nippon Hyoron Sha, Co., Ltd., 3-12-4 Minami-Otsuka, Toshima-ku, Tokyo, 170-8474, Japan}
\email{iino@nippyo.co.jp}
\address{The Kaisei Academy, 4-2-4 Nishi-Nippori, Arakawa-ku, Tokyo, 116-0013, Japan}
\email{sakiyama-ms@kaiseigakuen.jp}

\begin{document}
\maketitle
\begin{abstract}
For a planar integral lattice $L$, let $\nu(L)$ denote the square-free part of the integer $D(L)^2$, where 
$D(L)$ stands for the area of a fundamental parallelogram of $L$. 
For each odd integer $n$ with $3 \leq n<29$, 
a planar lattice $L$ contains an equilateral $n$-gon if and only if
$L$ is similar to an integral lattice $L'$ such that $\nu(L')\equiv 3 \pmod 4$ and the largest prime factor $p$ of $\nu(L')$ satisfies $p \leq n$. 
Moreover, such $L$ contains a convex equilateral $n$-gon, which answers a problem posed by Maehara.  
\end{abstract}

\section{Introduction}\label{introduction}
In 1973, Ball~\cite{Ball} considered the existence of equilateral polygons with vertices on the square lattice and
proved the following.

\begin{Thm}[{\cite[Theorems 1 and 6]{Ball}}]\label{Ball}
The square lattice $\mathbb{Z}^2$ does not contain an equilateral $n$-gon if $n$ is odd. 
The lattice $\mathbb{Z}^2$ contains a convex equilateral $n$-gon if $n$ is even.
\end{Thm}

Maehara~\cite{Educ, Maehara} considered the existence of equilateral polygons with vertices on a general planar lattice. 
He showed many interesting results, some of which are listed below. (The definitions of some terms will be explained in the next section.)

\begin{Thm}[{\cite[Theorem 5.1]{Maehara}}]\label{integral}
A planar lattice $L$ contains a convex equilateral $n$-gon for some $n\neq 4$ if and only if 
$L$ is similar to an integral lattice.
\end{Thm}

\begin{Thm}[{\cite[Theorem 4.1]{Maehara}}]\label{mod}
Every planar integral lattice $L$ contains a convex equilateral $n$-gon for every even $n \geq 4$.
A planar integral lattice $L$ contains an equilateral $n$-gon for some odd $n\geq 3$ if and only if $\nu(L)\equiv 3 \pmod 4$. 
\end{Thm}

\begin{Thm}[{\cite[Theorem 6.1]{Maehara}}]\label{triangle}
For a planar lattice $L$, the following three are equivalent. 
(i) $L$ contains an equilateral triangle. 
(ii) $L$ contains a convex equilateral $n$-gon for every $n\geq 3$. 
(iii) $L$ is similar to an integral lattice $L'$ with $\nu(L')=3$. 
\end{Thm}

From Theorems~\ref{integral} and \ref{mod}, we see that 
for each even $n \neq 4$, a planar lattice $L$ contains 
a convex equilateral $n$-gon
if and only if $L$ is similar to an integral lattice. 
In this paper, we consider the following problem. 

\begin{Prob}\label{oddgon}
Let $n$ be an odd integer with $n \geq 3$. Find the condition on a planar lattice $L$ so that $L$ contains an equilateral $n$-gon.
\end{Prob}

Theorem~\ref{triangle} answers Problem~\ref{oddgon} in the case where $n=3$. 
In this paper, we extend Theorem~\ref{triangle} and answer Problem~\ref{oddgon}
when
$n$ is small. 
We also give an affirmative answer to the following problem, which is posed by Maehara~\cite{Educ, Maehara}.

\begin{Prob}[{\cite[Problem 1.1]{Maehara}}]\label{MaeharaProblem}
Is there a planar integral lattice $L$ with $\nu(L)\neq 3$ that contains a convex equilateral $n$-gon for some odd $n>3$?
\end{Prob}

\section{Preliminaries}\label{preliminaries}

In this paper, an element of a Euclidean plane $\mathbb{R}^2$ is called a \textit{vector} as well as a \textit{point}. 
For linearly independent vectors $\vect{a}, \vect{b}$ in the plane $\mathbb{R}^2$, 
$L[\vect{a}, \vect{b}]$ denotes the \textit{planar lattice} generated by $\vect{a}, \vect{b}$, that is, 
$L[\vect{a}, \vect{b}]=\{ m\vect{a} + n\vect{b} \mid m, n \in \mathbb{Z} \} \subset \mathbb{R}^2$.
If a lattice $L$ is contained in another lattice $L'$, then $L$ is said to be a \textit{sublattice} of $L'$. 
A lattice $L$ is said to be
\textit{similar to} 
$L'$ if there is a $\lambda >0$ such that $\lambda L = \{ \lambda \vect{x} \mid \vect{x} \in L \}$ is isometric to $L'$. 
A planar lattice $L$ is called an \textit{integral lattice}, if for every $\vect{x}, \vect{y} \in L$, 
the inner product $\vect{x}\cdot \vect{y}$ is an integer. 

For a planar lattice $L$, let $D(L)$ denote the area of a fundamental parallelogram of the lattice $L$. 
Thus, if $L=L[\vect{a}, \vect{b}]$, then $D(L)=|\det(\vect{a}, \vect{b})|$. Although there are many choices of vectors $\vect{a}, \vect{b}$
that generate the planar lattice $L$, the area $D(L)=|\det(\vect{a}, \vect{b})|$ is independent of the choice of generating vectors 
$\vect{a}, \vect{b}$ of $L$. 
If $L$ is a planar integral lattice, then $D(L)^2$ is an integer. In fact, if $L=L[\vect{a}, \vect{b}]$, then
\[ D(L)^2 = \det(\vect{a}, \vect{b})^2 = \left| \begin{array}{cc} \vect{a}\cdot \vect{a} & \vect{a} \cdot \vect{b} \\
\vect{b}\cdot \vect{a} & \vect{b} \cdot \vect{b} \end{array} \right|  \in \mathbb{Z}. \]
Let us denote the square-free part of $D(L)^2$ by $\nu(L)$, that is, 
$\nu(L)$ is the square-free integer that satisfies $D(L)^2=k^2\nu(L)$ for some integer $k$.

For a square-free integer $m>0$, we denote the rectangular lattice $L[(1,0), \allowbreak (0,\sqrt{m})]$ by $\Lambda(m)$.

An \textit{equilateral polygon} is a polygon whose sides are all equal in length. 
If every vertex of a polygon $P$ is in a set $S$, then we say $S$ \textit{contains} the polygon $P$.

For an equilateral $n$-gon $\textrm{A}_1\textrm{A}_2\ldots\textrm{A}_n$, 
the vectors $\vect{e}_1=\overrightarrow{\textrm{A}_n\textrm{A}_1}, \, \allowbreak 
\vect{e}_2=\overrightarrow{\textrm{A}_1\textrm{A}_2}, \, \allowbreak 
\ldots ,  \, \allowbreak 
\vect{e}_n=\overrightarrow{\textrm{A}_{n-1}\textrm{A}_n}$
are called the \textit{edge-vectors} of the polygon.
Clearly $\vect{e}_1 + \vect{e}_2 + \cdots + \vect{e}_n =\vect{0}$
and $|\vect{e}_1|=|\vect{e}_2|=\cdots =|\vect{e}_n|$.

Maehara~\cite{Educ, Maehara}
proved the following results. We will use them in the next section.

\begin{Lem}[{\cite[Lemma 3]{Educ}}]\label{convex}
A planar lattice $L$ contains a convex equilateral $n$-gon if and only if $L$ contains $n$ distinct vectors 
$\vect{e}_1, \vect{e}_2, \ldots, \vect{e}_n$ such that $|\vect{e}_1|=|\vect{e}_2|=\cdots =|\vect{e}_n|$ and 
$\vect{e}_1+\vect{e}_2+\cdots +\vect{e}_n=\vect{0}$. 
\end{Lem}

\begin{Lem}[{\cite[Corollary 4.1]{Maehara}}]\label{similar}
A planar integral lattice $L$ contains an equilateral $n$-gon (resp. a convex equilateral $n$-gon) if and only if
$\Lambda(\nu(L) )$ contains an equilateral $n$-gon (resp. a convex equilateral $n$-gon).
\end{Lem}

\begin{Lem}[{\cite[Lemma 4.1]{Maehara}}]\label{Lemma4.1}
If $\Lambda(m)$ contains an equilateral $n$-gon (resp. a convex equilateral $n$-gon),
then it contains an equilateral $(n+2)$-gon (resp. a convex equilateral $(n+2)$-gon).
\end{Lem}

\section{Results}\label{results}

Let us consider the condition on a planar lattice $L$ so that $L$ contains an equilateral $n$-gon (or a convex equilateral $n$-gon) for each odd $n$. 
First, we prove the following theorem.

\begin{Thm}\label{prime}
Let $n\geq 3$ be an odd number.
If a planar integral lattice $L$ contains an equilateral $n$-gon, not necessarily convex,  
then $n \geq p$ for every prime factor $p$ of $\nu(L)$. 

\end{Thm}

\begin{proof}
Assume that $L$ contains an equilateral $n$-gon and $p$ is a prime factor of $\nu(L)$. By Lemma~\ref{similar}, 
$\Lambda(\nu )$ with $\nu=\nu(L)$ contains an equilateral $n$-gon.
Let $P$ be
one of the equilateral $n$-gons
contained in $\Lambda(\nu )$ and 
let $\vect{e}_1, \, \vect{e}_2, \, \ldots , \, \vect{e}_n$ be the edge-vectors of $P$. 
Let $\vect{e}_1=(s, \, t\sqrt{\nu })$, where $s, \, t \in \mathbb{Z}$. 
By the linear transformation $f$ induced by the matrix 
$\bigl(
\begin{smallmatrix}
   s & t\sqrt{\nu } \\
   -t\sqrt{\nu } & s
\end{smallmatrix}
\bigl)$,
the vector $\vect{e}_1$ is mapped to $f(\vect{e}_1)=(s^2+t^2\nu , \, 0)$.
Since $\bigl(
\begin{smallmatrix}
   s & t\sqrt{\nu } \\
   -t\sqrt{\nu } & s
\end{smallmatrix}
\bigl)
\bigl(
\begin{smallmatrix}
   a \\
   b\sqrt{\nu } 
\end{smallmatrix}
\bigl) 
=
\bigl(
\begin{smallmatrix}
  sa+tb\nu \\
   (-ta+sb)\sqrt{\nu } 
\end{smallmatrix}
\bigl)$, $f$ maps $\Lambda(\nu )$ into $\Lambda(\nu )$. 
Hence the vectors $f(\vect{e}_1), \, f(\vect{e}_2), \, \ldots , \, f(\vect{e}_n)$
are in $\Lambda(\nu )$.
Since $f$ is a similarity, the vectors $f(\vect{e}_1), \, f(\vect{e}_2), \, \ldots , \, f(\vect{e}_n)$
determine
an equilateral $n$-gon $P'$, which is similar to $P$. 
Since the side length of $P'$ is the integer $s^2+t^2\nu$,
there exists an equilateral $n$-gon contained in $\Lambda(\nu )$ whose side length is an integer.

Let $Q$ be one of the equilateral $n$-gons contained in $\Lambda(\nu )$ whose side length is the minimal integer,
and let $\vect{v}_1, \, \ldots,  \, \vect{v}_n$ be the edge-vectors of $Q$.
Let $\vect{v}_i=(a_i, \, b_i\sqrt{\nu })$ ($a_i, \, b_i \in \mathbb{Z}$) and let $k$ be the side length of $Q$.
Since each side has length $k$, we have
\begin{equation}\label{D}
{a_i}^2+{b_i}^2 \nu =k^2
\end{equation}
for each $i \in \{1, \, 2, \, \ldots, \, n \}$.
Assume that $k$ is a multiple of $p$. 
Since $k$ and $\nu $ are multiples of $p$, it follows from equation (\ref{D}) that 
${a_i}^2$ is a multiple of $p$ for each $i$.
Since $p$ is prime, $a_i$ is a multiple of $p$.
Now, since ${a_i}^2$ and $k^2$ are multiples of $p^2$,
it follows again from equation (\ref{D}) that ${b_i}^2\nu $ is a multiple of $p^2$. 
Since $\nu $ is square-free, $\nu $ is not a multiple of $p^2$. Hence
${b_i}^2$ must be a multiple of $p$, and so is $b_i$.
Now since $a_1, \, a_2, \, \ldots, \, a_n, \, b_1, \, b_2, \, \ldots, \, b_n, \, k$ are multiples of $p$, 
$(1/p)\vect{v}_i$ is a vector in $\Lambda(\nu )$ for each $i$, 
and hence an $n$-gon determined by the edge-vectors $(1/p)\vect{v}_1, \, \ldots , \, (1/p)\vect{v}_n$
can be contained in $\Lambda(\nu)$ and has side length $k/p$, which is an integer smaller than $k$.
This contradicts the minimality of $k$. 
Therefore $k$ is not a multiple of $p$. 

Since $\nu $ is a multiple of $p$, it follows from (\ref{D}) that 
${a_i}^2 \equiv k^2 \pmod p$, and hence
$(a_i-k)(a_i+k) \equiv 0 \pmod p$.
Since $p$ is prime, either $a_i \equiv k \pmod p$ or $a_i \equiv -k \pmod p$ holds.
Now let $n'=\# \{ i \mid a_i \equiv k \pmod p \}$. 
We have $\# \{ i \mid a_i \not\equiv k \pmod p \}=n-n'$.
Since $i \in \{ i \mid a_i \not\equiv k \pmod p \}$ implies $a_i \equiv -k \pmod p$, we have 
$\sum_{i=1}^n a_i \equiv n'k +(n-n')(-k) \equiv (2n'-n) k \pmod p$.
Since $\sum_{i=1}^n \vect{v}_i = \vect{0}$, we have $\sum_{i=1}^n a_i =0$,
and hence $\sum_{i=1}^n a_i \equiv 0 \pmod p$.
Now we have $(2n'-n) k \equiv 0 \pmod p$. 
Since $k$ is not a multiple of $p$, $2n'-n$ must be a multiple of $p$.
Since $0 \leq n' \leq n$, we have $-n \leq 2n'-n \leq n$. 
Since $n$ is odd, we have $2n'-n \neq 0$. 
Therefore, the interval $[ -n, \, n]$ contains a nonzero multiple of $p$, 
which shows $n\geq p$.
\end{proof}

From Theorems~\ref{mod} and \ref{prime}, 
we know that for each odd integer $n \geq 3$, for a planar integral lattice $L$ to contain an equilateral $n$-gon 
it is necessary that $\nu(L) \equiv 3 \pmod 4$ and the largest prime factor $p$ of $\nu(L)$ satisfies $p \leq n$.
Is this condition also sufficient? 
We will prove that the answer is yes if $n <17$.

\begin{Prop}\label{examples}
Let $m$ be a square-free positive integer with $m \equiv 3 \pmod 4$ and let $p$ be the largest prime factor of $m$.
If $p < 17$, then the rectangular lattice $\Lambda(m)$ contains a convex equilateral $p$-gon. 
\end{Prop}

\begin{proof}
Each square-free positive integer $m$ satisfying $p < 17$ is 
expressed as $m=2^a \cdot 3^b \cdot 5^c \cdot 7^d \cdot 11^e \cdot 13^f$, where $a, b, c, d, e, f \in \{0, 1\}$.
Since $2^a\cdot 3^b \cdot 5^c \cdot 7^d \cdot 11^e \cdot 13^f \equiv
2^a\cdot (-1)^b \cdot 1^c \cdot (-1)^d \cdot (-1)^e \cdot 1^f \equiv 2^a \cdot (-1)^{b+d+e} \pmod 4$, 
$m$ of this form satisfies $m \equiv 3 \pmod 4$ if and only if $a=0$ and $b+d+e \equiv 1 \pmod 2$.  
Hence it is possible to list all of the square-free positive integers $m$ satisfying $p < 17$ as in Table~\ref{13}. 
For each of such $m$, $\Lambda(m)$ contains $p$ distinct vectors of equal length with sum $\vect{0}$, 
which are illustrated in Table~\ref{13}.
From Lemma~\ref{convex}, we conclude that $\Lambda(m)$ contains a convex equilateral $p$-gon.
\begin{table}[htbp]

\centering
\caption{square-free $m$ with $m \equiv 3 \pmod 4$ and $p < 17$, and $p$ distinct vectors of equal length in $\Lambda(m)$ with sum $\vect{0}$}

\footnotesize
\begin{tabular}{r@{\quad }l@{\ }l@{\ }l@{\ }l@{\ }l@{\quad }r@{\quad }l}
$m$ & \multicolumn{5}{l}{prime factors} & $p$ & $p$ distinct vectors of equal length in $\Lambda(m)$ with sum $\vect{0}$ \\\hline
3 & 3 &  &  &  &  & 3 & $(2, 0)$,
$(-1, \pm \sqrt{m})$ \\[0.5pt]

15 & 3, & 5 & &  &  & 5 & $(483724, 0)$,
$(445129, 48887\sqrt{m})$,
$(-379901, 77315\sqrt{m})$,\\[-1.5pt] & & & & & & & \quad 
$(-483631, -2449\sqrt{m})$,
$(-65321, -123753\sqrt{m})$\\[0.5pt]

7 &  &  & 7 &  &  &  7 & $(88, 0)$,
$(81, \pm 13\sqrt{m})$,
$(-38, \pm 30\sqrt{m})$,
$(-87, \pm 5\sqrt{m})$\\[0.5pt]

35 &  & 5, & 7 &  &  &  7 & $(17226, 0)$,
$(13194, \pm 1872\sqrt{m})$,
$(-4726, \pm 2800\sqrt{m})$,\\[-1.5pt] & & & & & & & \quad 
$(-17081, \pm 377\sqrt{m})$\\[0.5pt]

11 &  &  &  & 11 &  &  11 &  $(90, 0)$,
$(57, \pm 21\sqrt{m})$,
$(35, \pm 25\sqrt{m})$,
$(-9, \pm 27\sqrt{m})$,\\[-1.5pt] & & & & & & & \quad 
$(-42, \pm 24\sqrt{m})$,
$(-86, \pm 8\sqrt{m})$\\[0.5pt]

55 &  & 5, &  & 11 &  &  11 & $(728, 0)$,
$(717, \pm 17\sqrt{m})$,
$(552, \pm 64\sqrt{m})$,
$(-273, \pm 91\sqrt{m})$,\\[-1.5pt] & & & & & & & \quad 
$(-658, \pm 42\sqrt{m})$,
$(-702, \pm 26\sqrt{m})$\\[0.5pt]

231 & 3, &  & 7, & 11 &  &  11 & $(800, 0)$,
$(569, \pm 37\sqrt{m})$,
$(415, \pm 45\sqrt{m})$,
$(-124, \pm 52\sqrt{m})$,\\[-1.5pt] & & & & & & & \quad 
$(-520, \pm 40\sqrt{m})$,
$(-740, \pm 20\sqrt{m})$\\[0.5pt]

1155 & 3, & 5, & 7, & 11 &  & 11 &  $(22678, 0)$,
$(22447, \pm 95\sqrt{m})$,
$(3967, \pm 657\sqrt{m})$,\\[-1.5pt] & & & & & & & \quad 
$(-5273, \pm 649\sqrt{m})$,
$(-15283, \pm 493\sqrt{m})$,\\[-1.5pt] & & & & & & & \quad 
$(-17197, \pm 435\sqrt{m})$\\[0.5pt]

39 & 3, &  &  &  & 13 & 13  & $(440, 0)$,
$(401, \pm 29\sqrt{m})$,
$(310, \pm 50\sqrt{m})$,
$(154, \pm 66\sqrt{m})$,\\[-1.5pt] & & & & & & & \quad 
$(-275, \pm 55\sqrt{m})$,
$(-392, \pm 32\sqrt{m})$,
$(-418, \pm 22\sqrt{m})$\\[0.5pt]

195 & 3, & 5, &  &  & 13 & 13 & $(1666, 0)$,
$(1601, \pm 33\sqrt{m})$,
$(1406, \pm 64\sqrt{m})$,
$(119, \pm 119\sqrt{m})$,\\[-1.5pt] & & & & & & & \quad 
$(-791, \pm 105\sqrt{m})$,
$(-1519, \pm 49\sqrt{m})$,
$(-1649, \pm 17\sqrt{m})$\\[0.5pt]

91 &  &  & 7, &  & 13 & 13 & $(5890, 0)$,
$(5877, \pm 41\sqrt{m})$,
$(3875, \pm 465\sqrt{m})$,\\[-1.5pt] & & & & & & & \quad 
$(-402, \pm 616\sqrt{m})$,
$(-1767, \pm 589\sqrt{m})$,
$(-5043, \pm 319\sqrt{m})$,\\[-1.5pt] & & & & & & & \quad 
$(-5485, \pm 225\sqrt{m})$\\[0.5pt]

455 &  & 5, & 7, &  & 13 & 13 & $(4104, 0)$,
$(3519, \pm 99\sqrt{m})$,
$(3064, \pm 128\sqrt{m})$,
$(646, \pm 190\sqrt{m})$,\\[-1.5pt] & & & & & & & \quad 
$(-2214, \pm 162\sqrt{m})$,
$(-3306, \pm 114\sqrt{m})$,
$(-3761, \pm 77\sqrt{m})$\\[0.5pt]

143 &  &  &  & 11, & 13 & 13 & $(6384, 0)$,
$(4902, \pm 342\sqrt{m})$,
$(3472, \pm 448\sqrt{m})$,\\[-1.5pt] & & & & & & & \quad 
$(-532, \pm 532\sqrt{m})$,
$(-2391, \pm 495\sqrt{m})$,
$(-2534, \pm 490\sqrt{m})$,\\[-1.5pt] & & & & & & & \quad 
$(-6109, \pm 155\sqrt{m})$\\[0.5pt]

715 &  & 5, &  & 11, & 13 & 13 & $(571778, 0)$,
$(492153, \pm 10885\sqrt{m})$,
$(424943, \pm 14307\sqrt{m})$,\\[-1.5pt] & & & & & & & \quad 
$(353157, \pm 16817\sqrt{m})$,
$(-472382, \pm 12048\sqrt{m})$,\\[-1.5pt] & & & & & & & \quad 
$(-526722, \pm 8320\sqrt{m})$,
$(-557038, \pm 4824\sqrt{m})$\\[0.5pt]

3003 & 3, &  & 7, & 11, & 13 & 13 & $(4496798, 0)$,
$(4468081, \pm 9259\sqrt{m})$,
$(4017631, \pm 36859\sqrt{m})$,\\[-1.5pt] & & & & & & & \quad 
$(-1091473, \pm 79605\sqrt{m})$,
$(-1638877, \pm 76415\sqrt{m})$,\\[-1.5pt] & & & & & & & \quad 
$(-3823202, \pm 43200\sqrt{m})$,
$(-4180559, \pm 30229\sqrt{m})$\\[0.5pt]

15015 & 3, & 5, & 7, & 11, & 13 & 13 &$(456688, 0)$,
$(446522, \pm 782\sqrt{m})$,
$(323828, \pm 2628\sqrt{m})$,\\[-1.5pt] & & & & & & & \quad 
$(21097, \pm 3723\sqrt{m})$,
$(-198122, \pm 3358\sqrt{m})$,\\[-1.5pt] & & & & & & & \quad 
$(-373297, \pm 2147\sqrt{m})$,
$(-448372, \pm 708\sqrt{m})$

\end{tabular}\
\label{13}
\end{table}
\end{proof}

\begin{Rem}\label{search}
The vectors illustrated in Table~\ref{13} can be found
in a quite simple search, using a computer, which we explain here.
To obtain $p$ distinct vectors of equal length contained in $\Lambda(m)$ with sum $\vect{0}$,  
first we search for $p$ distinct unit vectors contained in 
$\mathbb{Q} \times \sqrt{m}\mathbb{Q} = \{ (s, \, t\sqrt{m}) \mid s, \, t \in \mathbb{Q} \}$
with sum $\vect{0}$. 
For each positive integer $N$, let $U(m)_N$ be the set $\{ ({a}/{c}, {b\sqrt{m}}/{c}) \mid a, b, c \in \mathbb{Z}, 1 \leq c \leq N, a^2+b^2m=c^2 \}$, which is consisting of unit vectors in $\mathbb{Q}\times \sqrt{m}\mathbb{Q}$. 
For a fixed $N$, we use a computer to search for $p$ distinct vectors in $U(m)_N$
whose sum is $\vect{0}$. (An exhaustive search is possible since $U(m)_N$ is a finite set.)
By setting $N$ by $N=10^7$, we can find desired unit vectors. 
By multiplying them by an appropriate integer, 
we obtain $p$ distinct vectors of equal length in $\Lambda(m)$ with sum $\vect{0}$. 
\end{Rem}

Since $\nu(\Lambda(m))=m$ for a square-free positive integer $m$, 
Proposition~\ref{examples} answers Problem~\ref{MaeharaProblem} in the affirmative. 
It
also gives us the following corollary. 

\begin{Cor}\label{exist}
If a planar integral lattice $L$ satisfies $\nu(L) \equiv 3 \pmod 4$ and the largest prime factor $p$ of $\nu(L)$ satisfies 
$p < 17$, then $L$ contains a convex equilateral $n$-gon for every integer $n$ with $n \geq p$.   
\end{Cor}

\begin{proof}
Suppose that a planar integral lattice $L$ satisfies $\nu(L) \equiv 3 \pmod 4$ and the largest prime factor $p$ of $\nu(L)$ satisfies
$p < 17$. 
Then, from Proposition~\ref{examples}, $\Lambda(\nu(L))$ contains a convex equilateral $p$-gon. 
Hence by Lemma~\ref{Lemma4.1}, $\Lambda(\nu(L))$ contains a convex equilateral $n$-gon for every odd $n \geq p$.
By Lemma~\ref{similar},
$L$ contains a convex equilateral $n$-gon for every odd $n \geq p$.
Since $L$ contains a convex equilateral $n$-gon for every even $n$ by Theorem~\ref{mod}, 
the assertion follows. 
\end{proof}

Now we can prove the following theorem, which generalizes Theorem~\ref{triangle}. 

\begin{Thm}\label{generalized1.4} Let $n$ be an odd integer with $3 \leq n <17$. 
For a planar lattice $L$, the following three are equivalent. 
\begin{enumerate}
\item[(i)] $L$ contains an equilateral $n$-gon. 
\item[(ii)] $L$ contains a convex equilateral $k$-gon for every $k \geq n$. 
\item[(iii)] $L$ is similar to an integral lattice $L'$ such that $\nu(L')\equiv 3 \pmod 4$ and the largest prime factor $p$ of $\nu(L')$ satisfies $p \leq n$. 
\end{enumerate}
\end{Thm}

\begin{proof}
$\textrm{(ii)} \Rightarrow \textrm{(i)}$: Obvious. \\
$\textrm{(i)} \Rightarrow \textrm{(iii)}$: Suppose that (i) holds. 
Let 
$\vect{v}_1, \dots, \vect{v}_n$ be the edge-vectors of an equilateral $n$-gon contained in $L$. 
Suppose 
$\{ \vect{v}_1, \dots, \vect{v}_n \} \subset \{ \pm \vect{v}_1,  \pm \vect{v}_2\}$. 
Since $\vect{v}_1+\dots +\vect{v}_n=\vect{0}$, we have
$a\vect{v}_1+b(-\vect{v}_1)+c\vect{v}_2+d(-\vect{v}_2)=\vect{0}$ for some 
$a, b, c, d \in \{ 0, 1, \dots , n \}$ with $a+b+c+d=n$. 
Since $\vect{v}_1$ and $\vect{v}_2$ are linearly independent,
we have $a-b=c-d=0$, and hence $a+b+c+d$ is even,
which contradicts that $n$ is odd. 
Therefore, there exists $t \in \{3, \dots, n\}$ such that $\vect{v}_t \notin \{ \pm \vect{v}_1,  \pm \vect{v}_2\}$. 
Now $L$ contains six distinct vectors $\pm \vect{v}_1, \pm \vect{v}_2, \pm \vect{v}_t$
of equal length with sum $\vect{0}$, and hence $L$ contains a convex equilateral hexagon by Lemma~\ref{convex}.
Hence
by Theorem~\ref{integral}, 
$L$ is similar to an integral lattice $L'$. 
Since $L'$ contains an equilateral $n$-gon, $\nu(L')\equiv 3 \pmod 4$ from Theorem~\ref{mod}, and
the largest prime factor $p$ of $\nu(L')$ satisfies $p \leq n$ from Theorem \ref{prime}. \\
$\textrm{(iii)} \Rightarrow \textrm{(ii)}$:  Suppose that (iii) holds. Since $L'$ satisfies $\nu(L') \equiv 3 \pmod 4$ and the largest prime factor $p$ of $\nu(L')$
satisfies $p \leq n (<17)$, $L'$ contains a convex equilateral $k$-gon for every $k \geq p$ from Corollary~\ref{exist}. Since $p \leq n$, 
$L'$ contains a convex equilateral $k$-gon for every $k\geq n$. Since $L$ is similar to $L'$, $L$ contains a convex equilateral $k$-gon for every $k\geq n$. 
\end{proof}

Theorem~\ref{generalized1.4} answers Problem~\ref{oddgon} in the case where $n < 17$. 

\begin{Rem}
Proposition~\ref{examples} remains true if we replace ``If $p<17$'' by ``If $p<29$''.
To confirm this fact, it is sufficient to 
search for $p$ distinct unit vectors in $\mathbb{Q} \times \sqrt{m}\mathbb{Q}$ with sum $\vect{0}$
as in Remark~\ref{search} with setting $N$ by $N=10^{14}$. 
(Table~\ref{23} illustrates $p$ distinct vectors of equal length in $\Lambda(m)$ with sum $\vect{0}$ for some $m$.)
From this, we can see that Theorem~\ref{generalized1.4} holds for odd $n$ with $3 \leq n < 29$, which answers Problem~\ref{oddgon} in the case where $n < 29$. 
\begin{table}[htbp]
\centering
\caption{$p$ distinct vectors of equal length in $\Lambda(m)$ with sum $\vect{0}$ for some $m$ with $m\equiv 3 \pmod 4$ and $p<29$}
\footnotesize

\begin{tabular}{r@{\quad }l@{\ }l@{\ }l@{\ }l@{\ }l@{\ }l@{\ }l@{\ }l@{\quad }r@{\quad }l}
$m$ & \multicolumn{8}{l}{prime factors} & $p$ & $p$ distinct vectors of equal length in $\Lambda(m)$ with sum $\vect{0}$ \\\hline

51 & 3, &  &  &  &  & 17 & &  & 17 &
$(1430, 0)$,
$(1243, \pm 99\sqrt{m})$,
$(971, \pm 147\sqrt{m})$,\\[-1.5pt] & & & & & & & & & & \quad 
$(325, \pm 195\sqrt{m})$,
$(70, \pm 200\sqrt{m})$,\\[-1.5pt] & & & & & & & & & & \quad 
$(-406, \pm 192\sqrt{m})$,
$(-695, \pm 175\sqrt{m})$,\\[-1.5pt] & & & & & & & & & & \quad 
$(-1001, \pm 143\sqrt{m})$,
$(-1222, \pm 104\sqrt{m})$\\[0.5pt]

255255 & 3, & 5, & 7, & 11, & 13, & 17 & &  & 17 & $(41516416, 0)$,
$(37725824, \pm 34304\sqrt{m})$,\\[-1.5pt] & & & & & & & & & & \quad 
$(20300416, \pm 71680\sqrt{m})$,\\[-1.5pt] & & & & & & & & & & \quad 
$(16021856, \pm 75808\sqrt{m})$,\\[-1.5pt] & & & & & & & & & & \quad 
$(13864624, \pm 77456\sqrt{m})$,\\[-1.5pt] & & & & & & & & & & \quad 
$(7027156, \pm 80988\sqrt{m})$,\\[-1.5pt] & & & & & & & & & & \quad 
$(-34068899, \pm 46961\sqrt{m})$,\\[-1.5pt] & & & & & & & & & & \quad 
$(-40195019, \pm 20567\sqrt{m})$,\\[-1.5pt] & & & & & & & & & & \quad 
$(-41434166, \pm 5170\sqrt{m})$\\[0.5pt]

19 &  &  &  &  & & & 19  & &  19 & $(770, 0)$,
$(751, \pm 39\sqrt{m})$,
$(675, \pm 85\sqrt{m})$,\\[-1.5pt] & & & & & & & & & & \quad 
$(561, \pm 121\sqrt{m})$,
$(238, \pm 168\sqrt{m})$,\\[-1.5pt] & & & & & & & & & & \quad 
$(-275, \pm 165\sqrt{m})$, 
$(-427, \pm 147\sqrt{m})$,\\[-1.5pt] & & & & & & & & & & \quad 
$(-446, \pm 144\sqrt{m})$,
$(-693, \pm 77\sqrt{m})$,\\[-1.5pt] & & & & & & & & & & \quad 
$(-769, \pm 9\sqrt{m})$\\[0.5pt]

1616615 &  & 5, & 7, & 11, & 13, & 17, & 19  & &  19 & $(1306809216, 0)$,
$(1306710264, \pm 12648\sqrt{m})$,\\[-1.5pt] & & & & & & & & & & \quad 
$(1099481216, \pm 555520\sqrt{m})$,\\[-1.5pt] & & & & & & & & & & \quad 
$(971630064, \pm 687312\sqrt{m})$,\\[-1.5pt] & & & & & & & & & & \quad 
$(43693054, \pm 1027226\sqrt{m})$,\\[-1.5pt] & & & & & & & & & & \quad 
$(-138916611, \pm 1021977\sqrt{m})$,\\[-1.5pt] & & & & & & & & & & \quad 
$(-694638111, \pm 870573\sqrt{m})$,\\[-1.5pt] & & & & & & & & & & \quad 
$(-863551284, \pm 771420\sqrt{m})$,\\[-1.5pt] & & & & & & & & & & \quad 
$(-1071771296, \pm 588064\sqrt{m})$,\\[-1.5pt] & & & & & & & & & & \quad 
$(-1306041904, \pm 35216\sqrt{m})$\\[0.5pt]

23 & & & & & & &  & 23 & 23 & 
$(1872, 0)$,
$(1826, \pm 86\sqrt{m})$,
$(1803, \pm 105\sqrt{m})$,\\[-1.5pt] & & & & & & & & & & \quad 
$(1044, \pm 324\sqrt{m})$,
$(676, \pm 364\sqrt{m})$,\\[-1.5pt] & & & & & & & & & & \quad 
$(78, \pm 390\sqrt{m})$,
$(-336, \pm 384\sqrt{m})$,\\[-1.5pt] & & & & & & & & & & \quad 
$(-819, \pm 351\sqrt{m})$,
$(-911, \pm 341\sqrt{m})$,\\[-1.5pt] & & & & & & & & & & \quad 
$(-1072, \pm 320\sqrt{m})$,
$(-1509, \pm 231\sqrt{m})$,\\[-1.5pt] & & & & & & & & & & \quad 
$(-1716, \pm 156\sqrt{m})$\\[0.5pt]

111546435 & 3, & 5, & 7, & 11, & 13, & 17, & 19, & 23 & 23 & 
$(46923183273602, 0)$,\\[-1.5pt] & & & & & & & & & & \quad 
$(46918331222398, \pm 63889920\sqrt{m})$,\\[-1.5pt] & & & & & & & & & & \quad 
$(46179809577838, \pm 787694056\sqrt{m})$,\\[-1.5pt] & & & & & & & & & & \quad 
$(45535252265413, \pm 1072579179\sqrt{m})$,\\[-1.5pt] & & & & & & & & & & \quad 
$(32540987084087, \pm 3200887671\sqrt{m})$,\\[-1.5pt] & & & & & & & & & & \quad 
$(1360169729983, \pm 4440962273\sqrt{m})$,\\[-1.5pt] & & & & & & & & & & \quad 
$(-18098964883063, \pm 4099034551\sqrt{m})$,\\[-1.5pt] & & & & & & & & & & \quad 
$(-21535244754542, \pm 3947292712\sqrt{m})$,\\[-1.5pt] & & & & & & & & & & \quad 
$(-33327326200258, \pm 3127514608\sqrt{m})$,\\[-1.5pt] & & & & & & & & & & \quad 
$(-36400586649517, \pm 2803613027\sqrt{m})$,\\[-1.5pt] & & & & & & & & & & \quad 
$(-41444796116242, \pm 2083271992\sqrt{m})$,\\[-1.5pt] & & & & & & & & & & \quad 
$(-45189222912898, \pm 1196604960\sqrt{m})$

\end{tabular}\
  \label{23}
\end{table}
\end{Rem}

For general $n$, Problem~\ref{oddgon} is unsolved.

\subsection*{Acknowledgement}
The authors would like to thank Hiroshi Maehara and the referee for their helpful advice.

\end{document}